\documentclass{article}
\usepackage{graphicx} % Required for inserting images
\usepackage[USenglish]{babel}
\usepackage{bbm}
\usepackage{amsmath}
\usepackage{amsfonts}
\usepackage{amssymb}
\usepackage{amsthm}
\usepackage{bussproofs}
\usepackage{enumerate}
\usepackage{enumitem}
\usepackage{mathtools}
\usepackage[hidelinks]{hyperref}
\usepackage{scrextend}
\usepackage{tikz}
\usepackage[autostyle=false, style=english]{csquotes}
\usepackage{commath}

\makeatletter
\newcommand*{\rn}[1]{%
  \expandafter\@rn\csname c@#1\endcsname%
}

\newcommand*{\@rn}[1]{%
  $\ifcase#1\or(i)\or(ii)\or(iii)\or(iv)\or(v)\or(vi)\or(vii)\or(viii)\or(ix)\or(x)%
    \else\@ctrerr\fi$%
}
\AddEnumerateCounter{\rn}{\@rn}{53.13}
\makeatother

  \newbox\gnBoxA
\newdimen\gnCornerHgt
\setbox\gnBoxA=\hbox{$\ulcorner$}
\global\gnCornerHgt=\ht\gnBoxA
\newdimen\gnArgHgt
\def\Godelnum #1{%
\setbox\gnBoxA=\hbox{$#1$}%
\gnArgHgt=\ht\gnBoxA%
\ifnum     \gnArgHgt<\gnCornerHgt \gnArgHgt=0pt%
\else \advance \gnArgHgt by -\gnCornerHgt%
\fi \raise\gnArgHgt\hbox{$\ulcorner$} \box\gnBoxA %
\raise\gnArgHgt\hbox{$\urcorner$}}

\makeatletter
\newcommand{\pushright}[1]{\ifmeasuring@#1\else\omit\hfill$\displaystyle#1$\fi\ignorespaces}
\newcommand{\pushleft}[1]{\ifmeasuring@#1\else\omit$\displaystyle#1$\hfill\fi\ignorespaces}
\makeatother

\newcommand{\PP}{\mathbb{P}}

\newcommand{\Q}{\dot{\mathbb{Q}}}
\newcommand{\1}{\mathbbm{1}}
\newcommand{\forces}{\Vdash}
\newcommand{\res}{\upharpoonright}

\newcommand{\dotieconcat}[2]{\text{\raisebox{.8ex}{$\smallfrown$}}}

\newcommand{\QQ}{\mathbb{Q}}
\newcommand{\BB}{\mathbb{B}}

\newcommand{\ZFC}{\mathrm{ZFC}}

\newtheoremstyle{nopoint}% name of the style to be used
  {}{}{\itshape}{}{\bfseries}{}{5pt}{}% Manually specify head

\theoremstyle{plain}
\newtheorem{thm}{Theorem}[section]
\newtheorem{prop}[thm]{Proposition}
\newtheorem{lemm}[thm]{Lemma}
\newtheorem{fact}[thm]{Fact}

\newtheorem{cor}[thm]{Corollary}
\newtheorem{claim}[thm]{Claim}
\newtheorem{sub}[thm]{Subclaim}
\theoremstyle{definition}
\newtheorem{defn}[thm]{Definition}

\newtheorem{rem}[thm]{Remark}

\newtheorem*{que*}{Question}

\theoremstyle{nopoint}
\newtheorem*{lemm*}{Lemma}
\newtheorem*{thm*}{Theorem}
\newtheorem*{rem*}{Remark}
\usepackage{mathabx}
\usepackage{tikz-cd}
\usepackage{tikz}\usetikzlibrary{arrows}

\newcommand{\plus}{+}

\newcommand{\diacol}{\diamondsuit(\omega_1^{{<}\omega})}

\newcommand{\diacolp}{\diamondsuit^{\plus}(\omega_1^{{<}\omega})}

\newcommand{\diab}{\diamondsuit(\mathbb B)}
\newcommand{\dia}[1]{\diamondsuit(#1)}

\newcommand{\diaba}{\diamondsuit^{\plus}(\BB)}

\newcommand{\good}{\text{-slim}}

\newcommand{\goodit}{\text{\textit{-slim}}}
\newcommand{\fgood}{f\good}
\newcommand{\fgoodit}{f\goodit}

\newcommand{\Add}{\mathrm{Add}}

\newcommand{\NS}{\mathrm{NS}_{\omega_1}}
\newcommand{\NSf}{\mathrm{NS}_{f}}
\newcommand{\NScheckf}{\mathrm{NS}_{\check{f}}}

\newcommand{\fMMpp}{\mathrm{MM}^{++}(f)}

\newcommand{\BP}{\mathbb{B}}
\newcommand{\MM}{\mathrm{MM}}
\newcommand{\MMpp}{\MM^{++}}

\newcommand{\lh}{\mathrm{lh}}

\newcommand{\Ord}{\mathrm{Ord}}

\newcommand{\Htwo}{H_{\omega_2}}

\newcommand{\Lim}{\mathrm{Lim}}

\newcommand{\sucm}{\mathrm{suc}}
\newcommand{\msup}[2]{#1^{(#2)}}
\newcommand{\pred}{\mathrm{pred}}
\newcommand{\hooks}{\angle\ }
\newcommand{\standarditeration}{\langle \PP_\alpha,\Q_\beta\mid\alpha\leq\gamma, \beta<\gamma\rangle}

\newcommand{\Pomo}{\mathcal P(\omega_1)}

\newcommand{\Qmax}{\mathbb Q_{\mathrm{max}}}
\newcommand{\SRP}{\mathrm{SRP}}

\newcommand{\RCS}{\mathrm{RCS}}
\newcommand{\nicelim}{\mathrm{nicelim}}

\newcommand{\todorcevic}{Todor\v{c}evi\'{c}}

\numberwithin{thm}{section}

\theoremstyle{plain}

\theoremstyle{definition}
\newtheorem*{defn*}{Definition}

\newtheorem{conv}[thm]{Convention}

\newcommand{\CS}{\mathrm{CS}}

\title{An Iteration Theorem for $\omega_1$-preserving Forcings}
\author{Andreas Lietz\footnote{Institut f\"ur Mathematische Logik und Grundlagenforschung, Universit\"at M\"unster, Einsteinstrasse 62, 48149 M\"unster, FRG. }\hspace{5pt}\footnote{Current address: Institut für Diskrete Mathematik und Geometrie, TU Wien, Wiedner Hauptstrasse 8-10/104, 1040 Wien, AT\\
This paper is part of the authors PhD thesis.}}
\date{}

\begin{document}
\maketitle

\begin{abstract}
We prove an iteration theorem which guarantees for a wide class of nice iterations of $\omega_1$-preserving forcings that $\omega_1$ is not collapse, at the price of needing large cardinals to burn as fuel. More precisely, we show that a nice iteration of $\omega_1$-preserving forcings which force $\SRP$ at successor steps and preserves old stationary sets does not collapse $\omega_1$.
\end{abstract}

\section{Introduction}
The method of iterated forcing is a powerful yet flexible tool in establishing independence results. Say, the goal is to produce a forcing extension of the universe with a specific property. Frequently, it is the case that it is much easier to find a forcing $\PP$, which solves this problem for ``a single instance" or ''all instances in $V$", but may add new ``unresolved instances" at the same time. One can then hope to iterate $\PP$ up to some closure point, usually a sufficiently large regular cardinal $\kappa$ so that the whole iteration is $\kappa$-c.c., so that in end all instances have been dealt with and the full desired property holds. This can only work if the iteration in question preserves the progress of earlier stages up until the end. Theorems which guarantee such a preservation are often called iteration theorems. If the property in question is one about $\Htwo$ then at the very least it is required that $\omega_1$ is preserved or maybe the somewhat stronger property that stationary sets are preserved. We give some examples.

\subsection{Iterations of c.c.c. Forcings}
The earliest iteration theorem is due to Solovay-Tennenbaum.

\begin{thm}[Solovay-Tennenbaum, \cite{solten}]
Suppose $\standarditeration$ is a finite support iteration of c.c.c. forcings. Then $\PP_\gamma$ is c.c.c..
\end{thm}

A Suslin tree is a tree of height $\omega_1$ with no uncountable chains and antichains. Forcing with a Suslin tree $T$ (with the reverse order) is a c.c.c. forcing and in the forcing extension $T$ is no longer Suslin. Not being a Suslin tree is a $\Sigma_1(\omega_1)$-property and so is upwards absolute to forcing extensions preserving $\omega_1$. As c.c.c. forcings preserve $\omega_1$, iterating forcing with Suslin trees produces models in which there are no Suslin tree and hence in which Suslin's hypothesis holds.

\begin{thm}[Solovay-Tennenbaum, \cite{solten}]
There is a c.c.c. forcing $\PP$ so that $V^\PP\models``\text{Suslin's hypothesis}"$.
\end{thm}

\subsection{Iterations of Proper Forcings}
The class of forcings with the countable chain condition is rather small, so not suitable in all cases. Shelah discovered the beautiful notion of proper forcing which is large enough to include both c.c.c. and $\sigma$-closed forcing, but nonetheless all such forcings preserve $\omega_1$.

\begin{defn}[Shelah,\cite{shelahbook}]
 A forcing $\PP$ is \textbf{proper} if for any large enough regular $\theta$ and any countable $X\prec H_\theta$ with $\PP\in X$, whenever $p\in \PP\cap X$ then there is some $q\leq p$ with 
 $$q\forces \check X\cap\Ord=\check X[\dot G]\cap \Ord.$$
\end{defn}

Shelah proved a famous iteration theorem for proper forcings. Though, as finite support iterations of non-c.c.c. forcings usually collapse $\omega_1$, the preferred support in this instance is countable support.

\begin{thm}[Shelah,\cite{shelahbook}]
Suppose $\standarditeration$ is a countable support iteration of proper forcings. Then $\PP_\gamma$ is proper.
\end{thm}

An Aronszajn tree $T$ is a tree of height $\omega_1$ with all countable levels and no cofinal branch. For a tree $T$ of height $\omega_1$ and $A\subseteq \omega_1$, let $T\res A$ denote the tree with nodes of a level $\alpha$ of $T$ with $\alpha\in A$ and the tree order inherited from $T$. Two trees $S, T$ of height $\omega_1$ are club-isomorphic iff there is a club $C\subseteq\omega_1$ so that $S\res C\cong T\res C$ as partial orders. Given two Aronszajn trees $S, T$, Abraham-Shelah discovered a proper forcing $\PP(T, S)$ which forces $S$ and $T$ to be club-isomorphic. Note that the property ``$S, T$ are club-isomorphic" is $\Sigma_1(S, T, \omega_1)$ and thus upwards-absolute to any $\omega_1$-preserving forcing extension. 

\begin{thm}[Abraham-Shelah, \cite{abrshelahtrees}]\label{semiproperiterationthm}
There is a proper forcing $\PP$ so that 
$$V^\PP\models``\text{Any two Aronszajn trees }S, T\text{ are club-isomorphic}".$$
\end{thm}

We remark that Suslin's hypothesis is an immediate consequence of ``any two Aronszajn trees are club-isomorphic". There provably is an Aronszajn tree which is \textit{special}, i.e. the union of countably many antichains. Such a tree cannot be club-isomorphic to a Suslin tree. 

\subsection{Iterations of Semiproper Forcings}

Later, Shelah proved another iteration theorem for the even larger class of semiproper forcings.

\begin{defn}[Shelah, \cite{shelahbook}]
 A forcing $\PP$ is \textbf{semiproper} if for any large enough regular $\theta$ and any countable $X\prec H_\theta$ with $\PP\in X$, whenever $p\in \PP\cap X$ then there is some $q\leq p$ with 
 $$q\forces \check X\cap\omega_1=\check X[\dot G]\cap \omega_1.$$
\end{defn}

From now on, we will denote $X\subseteq Y\wedge X\cap\omega_1=Y\cap\omega_1$ by $X\sqsubseteq Y$. So for example above we have $q\forces \check X\sqsubseteq \check X[\dot G]$.

\begin{thm}[Shelah]\label{semiproperiterationthm}
Suppose $\standarditeration$ is a $\RCS$-iteration of semiproper forcings. Then $\PP_\gamma$ is semiproper.
\end{thm}

Once again, the notion of support had to be changed. In the argument of Theorem \ref{semiproperiterationthm} it is crucial that if $\alpha<\gamma$ and $G_\alpha$ is $\PP_\alpha$-generic over $V$, then the tail iteration
$\langle \PP_{\alpha, \xi}, \Q_{\beta}\mid\xi\leq\gamma,\beta<\gamma\rangle$ is still a $\RCS$-iteration. This can fail for countable support iterations as, unlike proper forcings, semiproper forcings can turn regular cardinals into cardinals of countable cofinality. In fact, Theorem \ref{semiproperiterationthm} fails if $\RCS$-support is replaced with countable support.

Suppose $\mathcal I$ is an ideal on $\omega_1$. An $\mathcal I$-antichain is a set $\mathcal A\subseteq\mathcal P(\omega_1)-\mathcal I$ so that $S\cap T\in\mathcal I$ for any $S\neq T\in\mathcal A$. The ideal $\mathcal I$ is saturated if for all $\mathcal I$-antichains $\mathcal A$ we have $\vert\mathcal A\vert{\leq}\omega_1$. 

\begin{thm}[Shelah, see \cite{nssat} for a proof]\label{nssatthm}
Assume there is a Woodin cardinal. Then there is a semiproper forcing $\PP$ so that 
$$V^\PP\models``\NS\text{ is saturated}".$$
\end{thm}

If $\mathcal A$ is a maximal $\NS$-antichain then the sealing forcing $\mathcal S_{\mathcal A}$ is a natural stationary-set-preserving forcing which turns $\mathcal A$ into a maximal $\NS$-antichain of size ${\leq}\omega_1$ and the statement ``$\mathcal A$ is a maximal antichain of size ${\leq}\omega_1$" turns out to be $\Sigma_1(\mathcal A, \omega_1)$. Now, an instance of the sealing forcing is not semiproper in general, but Shelah shows that when iterating up to a Woodin cardinal and using a sealing forcing only when it is semiproper, it can be arranged that often enough sealing forcings are semiproper that in the end, $\NS$ is saturated. 

\subsection{Iterations of Stationary-Set-Preserving Forcing}

So what are the limits of iteration theorems? We have 
$$\text{c.c.c}\Rightarrow\text{proper}\Rightarrow\text{semiproper}\Rightarrow\text{stationary set preserving}.$$
and none of the implications can be reversed. However, while there are always non-c.c.c. proper forcings and non-proper semiproper forcings, consistently the class of semiproper forcings can agree with the class of stationary set preserving forcing, so these two notions are quite close. Nonetheless, there is no analogue of Theorem \ref{semiproperiterationthm} for stationary set preserving forcings. Consistently, a counterexample can be given along the lines of the discussion of Theorem \ref{nssatthm}. In the argument, the Woodin cardinal is used solely to verify that instances of sealing forcing are semiproper often enough, an inaccessible cardinal would suffice otherwise. But a Woodin cardinal is indeed required for the conclusion.

\begin{thm}[Steel, Jensen-Steel \cite{JensenSteelKWithoutAMeasurable}]\label{innermodelwithwoodinfromsatidealthm}
Suppose that there is a normal saturated ideal on $\omega_1$. Then there is an inner model with a Woodin cardinal.
\end{thm}

So suppose we work in an model without an inner model with a Woodin cardinal, say $V=L$, and there is an inaccessible cardinal. One could then try to iterate instances of the Sealing forcing along a suitable bookkeeping up to $\kappa$. In light of Theorem \ref{innermodelwithwoodinfromsatidealthm}, this cannot result in a forcing extension in which $\NS$ is saturated. It follows that the iteration collapses $\omega_1$ at some point, yet instances of the sealing forcing are always stationary set preserving.

A much more serious example is due to Shelah.

\begin{thm}[Shelah \cite{shelahbook}]
There is a full support iteration $$\coloneqq \langle \PP_n, \Q_m\mid n\leq\omega, m<\omega\rangle$$ of stationary set preserving forcings so that $\PP_{\omega}$ collapses $\omega_1$.
\end{thm}

In fact, in the above example it does not matter at all which kind of limit is taken, though we want to mention that countable support, $\RCS$ and full support iterations agree on length $\omega$ iterations.  The first forcing in Shelah's example is semiproper, but all subsequent forcings are not semiproper in the relevant extension. Semiproper forcing is the correct \textbf{regularity} property for stationary set preserving forcings in terms of iterations in the sense that 

\begin{enumerate}
    \item all semiproper forcings are stationary set preserving,
    \item consistently, all stationary set preserving forcings are semiproper and
    \item semiproper forcings can be iterated.
\end{enumerate}

We will define the class of \textbf{respectful} forcing which, in a slightly weaker sense, is a regularity property corresponding to the wider class of $\omega_1$-preserving forcings. 

\subsection{Iterations of $\omega_1$-Preserving Forcings}

When iterating $\omega_1$-preserving forcing which kill stationary sets there is another threat to preserving $\omega_1$ in the limit as illustrated in the following folklore example: For $S\subseteq\omega_1$ stationary, the club shooting forcing $\CS(S)$ is the canonical forcing that shoots a club trough $S$. Conditions are closed countable sets $c\subseteq S$ ordered by $d\leq_{\CS(S)}c$ iff $d\cap (\max(c)+1)=c$. If $G$ is generic for $\CS(S)$ then $\bigcup G$ is a club contained in $S$, so $\omega_1-S$ is nonstationary in $V[G]$, but $\omega_1$ is not collapsed, that is $\omega_1^{V[G]}=\omega_1^V$.
Now suppose $\langle S_n\mid n<\omega\rangle$ be a partition of $\omega_1$ into stationary sets. Let $\PP$ be a length $\omega$ iteration of the forcings $\CS(\omega_1-S_n)$ (it does not matter which limit we take at $\omega$). Then $\PP$ must collapse $\omega_1$ as in $V^\PP$, $\omega_1^V=\bigcup_{n<\omega} S_n$ is a countable union of nonstationary sets and hence must be nonstationary itself. Clearly, this is only possible if $\omega_1^V<\omega_1^{V^\PP}$.

The issue here does not stem from a lack of regularity of the forcings we used. In fact, for a stationary set $S\subseteq\omega_1$, the club shooting $\CS(S)$ is $S$-proper. The problem is much more that at each step of the iteration, we come back to $V$ to kill an ``old" stationary set. If we avoid the two presented issues of 
\begin{enumerate}
    \item using too many forcings lacking regularity properties and
    \item killing old stationary sets
\end{enumerate}
then we can prove an iteration theorem for $\omega_1$-preserving forcings. Without defining respectful forcings, a special case of our main result can be stated as follows.

\begin{thm}\label{omega1presiterationthm}
Suppose $\standarditeration$ is a nice iteration of $\omega_1$-preserving forcings so that 
\begin{enumerate}[label=$(\roman*)$]
    \item if $\alpha+2<\gamma$ then $\forces_{\PP_{\alpha+2}}``\text{Strong Reflection Principle}"$ and
    \item if $\alpha<\gamma$ then $\Q_\alpha$ is forced to preserve old stationary sets, i.e.
    $$\forall\beta<\alpha\ \forces_{\PP_{\alpha+1}}\NS\cap V[\dot G_{\beta}]=\NS^{V[\dot G_\alpha]}\cap V[\dot G_{\beta}].$$
\end{enumerate}
Then $\PP_\gamma$ preserves $\omega_1$. Moreover, we have for all $\alpha+1\leq\gamma$
$$\forall\beta\leq\alpha\ \forces_{\PP_\gamma}\NS\cap V[\dot G_\beta]=\NS^{V[\dot G_{\beta+1}]}\cap V[\dot G_\beta].$$
\end{thm}
In fact we will prove something more general which allows, e.g. the preservation of a Suslin tree on the side.

Here, the Strong Reflection Principle is the reflection principle isolated by \todorcevic.

\begin{defn}[Todor\v{c}evi\'{c}, \cite{todorcevicsrp}]
\hspace{1pt}
\begin{enumerate}[label=\rn*]
\item For $\theta$ an uncountable cardinal and $\mathcal S\subseteq [H_\theta]^\omega$ we define \index{Sperp@$\mathcal S^\perp$}
$$\mathcal S^\perp=\{X\in[H_\theta]^\omega\mid\forall Y\in [H_\theta]^\omega(X\sqsubseteq Y\rightarrow Y\notin S)\}.$$
\item The Strong Reflection Principle ($\SRP$)\index{Strong Reflection Principle} holds if: Whenever $\theta\geq\omega_2$ is regular, $a\in H_\theta$ and $S\subseteq [H_\theta]^\omega$ then $\mathcal S\cup \mathcal S^\perp$ contains a continuous increasing $\omega_1$-chain of countable elementary substructures of $H_\theta$ containing $a$, i.e.~there is $\langle X_\alpha\mid\alpha<\omega_1\rangle$ so that for all $\alpha<\omega_1$
\begin{enumerate}[label=$(\vec X.\roman*)$]
\item\label{SRPcond1} $X_\alpha\prec H_\theta$ is countable,
\item $X_\alpha\in \mathcal S\cup \mathcal S^\perp$,
\item $a\in X_0$,
\item $X_\alpha\in X_{\alpha+1}$ and
\item\label{SRPcond5} if $\alpha\in\Lim$ then $X_\alpha=\bigcup_{\beta<\alpha} X_\beta$.
\end{enumerate}
\end{enumerate}

\end{defn}

We note that $\SRP$ can always be forced assuming large cardinals.

\begin{thm}[Shelah]
Suppose there is a supercompact cardinal. Then there is a semiproper forcing $\PP$ so that $V^\PP\models\SRP$.
\end{thm}

As a consequence of this, assuming large cardinals, Theorem \ref{omega1presiterationthm} can be understood as a \textit{strategic} iteration theorem. Consider the following two player game $\mathrm{IG}_\gamma$ of length $\gamma$. 

\begin{center}
    \begin{tabular}{c||c|c|c|c|c|c|c|c}
        Player $I$ & $\Q_0$ &  & $\Q_2$ & &\dots&$\Q_\omega$ & &\dots \\
        \hline
        Player $II$& & $\Q_1$ & &$\Q_3$ &\dots & &$\Q_{\omega+1}$&\dots \\
         \end{tabular}
\end{center}

Player $I$ plays at all even stages, including limit steps. Player $I$ and $II$ cooperate in this way to produce an $\RCS$-iteration $\standarditeration$ of forcings which do not kill old stationary sets. Player $II$ wins iff $\PP_\gamma$ preserves $\omega_1$.

\begin{cor}
    Suppose there is a proper class of supercompact cardinals. Then for any $\gamma$, Player $II$ has a winning strategy for the game $\mathrm{IG}_\gamma$.
\end{cor}

\subsection*{Acknowledgements}
The author thanks his PhD advisor Ralf Schindler for many fruitful discussions as well as his guidance and support during the completion of this project.\\
Funded by the Deutsche Forschungsgemeinschaft (DFG, German Research
Foundation) under Germany’s Excellence Strategy EXC 2044 390685587, Mathematics M\"unster:
Dynamics - Geometry - Structure.
\section{Notation}

First, we fix some notation. We will extensively deal with countable elementary substructures $X\prec H_\theta$ for large regular $\theta$. We will make frequent use of the following notation:

\begin{defn}
Suppose $X$ is any extensional set.
\begin{enumerate}[label=\rn*]
\item $M_X$ denotes the transitive isomorph of $X$\index{MX@$M_X$}.
\item  $\pi_X\colon M_X\rightarrow X$ denotes the inverse collapse\index{piX@$\pi_X$}.
\item\label{deltadefn} $\delta^X\coloneqq\omega_1\cap X$\index{deltaX@$\delta^X$}.
\end{enumerate}
\end{defn} 

In almost all cases, we will apply this definition to a countable elementary substructure $X\prec H_\theta$ for some uncountable cardinal $\theta$. In some cases, the $X$ we care about lives in a generic extension of $V$, even though it is a substructure of $H_\theta^V$. In that case, $\delta^X$ will always mean $X\cap \omega_1^V$.\medskip

We will also sometimes make use of the following convention in order to ``unclutter" arguments.

\begin{conv}
If $X\prec H_\theta$ is an elementary substructure and some object $a$ has been defined before and $a\in X$ then we denote $\pi_X^{-1}(a)$ by $\bar a$\index{abar@$\bar a$ (short for $\pi_X^{-1}(a)$)}. 
\end{conv}

We will make use of this notation only if it is unambiguous.

\begin{defn}
If $X, Y$ are sets then $X\sqsubseteq Y$\index{<squaresubset@$\sqsubseteq$} holds just in case 
\begin{enumerate}[label=\rn*]
\item $X\subseteq Y$ and
\item $\delta^X=\delta^Y$.
\end{enumerate}
\end{defn}

We use the following notions of clubs and stationarity on $[H_\theta]^\omega$:

\begin{defn}
Suppose $A$ is an uncountable set. 
\begin{enumerate}[label=\rn*]
\item $[A]^\omega$\index{[A]omega@$[A]^\omega$} is the set of countable subsets of $A$.
\item $\mathcal C\subseteq [A]^\omega$ is a club in $[A]^\omega$\index{[A]omega@$[A]^\omega$!club in} if
\begin{enumerate}[label=$\alph*)$]
\item for any $X\in[A]^\omega$ there is a $Y\in \mathcal C$ with $X\subseteq Y$ and
\item if $\langle Y_n\mid n<\omega\rangle$ is a $\subseteq$-increasing sequence of sets in $\mathcal C$ then $\bigcup_{n<\omega} Y_n\in \mathcal C$.
\end{enumerate}
\item $\mathcal S\subseteq[A]^\omega$ is stationary in $[A]^\omega$\index{[A]omega@$[A]^\omega$!stationary in} if $\mathcal S\cap\mathcal C\neq\emptyset$ for any club $\mathcal C$ in $[A]^\omega$.
\end{enumerate}
\end{defn}

Next, we explain our notation for forcing iterations.
\begin{defn} Suppose $\PP=\standarditeration$ is an iteration and $\beta\leq\gamma$. We consider elements of $\PP$ as functions of domain (or length) $\gamma$.
\begin{enumerate}[label=$(\roman*)$]
\item If $p\in\PP_\beta$ then $\lh(p)=\beta$.
\item If $G$ is $\PP$-generic then $G_\beta$ denotes the restriction of $G$ to $\PP_\beta$, i.e.
$$G_\beta=\{p\res\beta\mid p\in G\}.$$
Moreover, $\dot G_\beta$ is the canonical $\PP$-name for $G_\beta$.
\item If $G_\beta$ is $\PP_\beta$-generic then $\PP_{\beta,\gamma}$ denotes (by slight abuse of notation) the remainder of the iteration, that is 
$$\PP_{\beta,\gamma}=\{p\in\PP_\gamma\mid p\res\beta\in G_\beta\}.$$
$\dot\PP_{\beta,\gamma}$ denotes a name for $\PP_{\beta,\gamma}$ in $V$.
\item If $G$ is $\PP$-generic and $\alpha<\beta$ then $G_{\alpha,\beta}$ denotes the projection of $G$ onto $\PP_{\alpha,\beta}$.
\end{enumerate}
\end{defn}

There will be a number of instances were we need a structure to satsify a sufficiently large fragment of $\ZFC$. For completeness, we make this precise.

\begin{defn}
\textit{Sufficiently much of }$\ZFC$\index{sufficiently much of ZFC@Sufficiently much of $\ZFC$} is the fragment $\ZFC^{-}+``\omega_1\text{ exists}"$. Here, $\ZFC^{-}$ is $\ZFC$ without the powerset axiom and with the collection scheme instead of the replacement scheme.
\end{defn}

\section{$\diab$ and $\diaba$}\label{diabanddiabasection}
We will introduce the combinatorial principle which will parameterize the main iteration theorem. These are generalizations of the principles $\diacol$ and $\diacolp$ isolated by Woodin \cite{woodinbook} in his study of $\Qmax$ \cite[Section 6.2]{woodinbook}. Most results in this Section are essentially due to Woodin and proven in \cite[Section 6.2]{woodinbook}.
% However, the notion of $f$-stationarity we define is both new and central to the theory we will develop.

\begin{defn}
Suppose $\BP\subseteq\omega_1$ is a forcing.
\begin{enumerate}[label=\rn*]
\item We say that $f$ \textit{guesses $\BP$-filters}\index{f guesses B filters@$f$ guesses $\BP$-filters} if $f$ is a function 
$$f\colon\omega_1\rightarrow H_{\omega_1}$$
and for all $\alpha<\omega_1$, $f(\alpha)$ is a $\BP\cap\alpha$-filter\footnote{We consider the empty set to be a filter.}.
\item Suppose $\theta\geq\omega_2$ is regular and $X\prec H_\theta$ is an elementary substructure. We say $X$ is $\fgoodit$\index{Set!f slim@$\fgood$}\footnote{We use the adjective ``slim" for the following reason: An $\fgood$ $X\prec H_\theta$ cannot be too fat compared to its height below $\omega_1$, i.e.~$\delta^X$.  If $X\sqsubseteq Y\prec H_\theta$ and $Y$ is $\fgood$ then $X$ is $\fgood$ as well, but the converse can fail.} if
\begin{enumerate}[label=$(X.\roman*)$]
 \item $X$ is countable,
 \item $f, \BP\in X$ and
 \item $f(\delta^X)$ is $\BB\cap\delta^X$-generic over $M_X$.
 \end{enumerate} 
\end{enumerate}
\end{defn}

\begin{defn}\label{diadefn}
Let $\BP\subseteq\omega_1$ be a forcing. $\diab$\index{diamondb@$\diab$} states that there is a function $f$ so that
\begin{enumerate}[label=\rn*]
 \item $f$ guesses $\BP$-filters and
 \item for any $b\in\BP$ and regular $\theta\geq\omega_2$
 \begin{align*}
 \{X\prec H_\theta\mid X \text{ is }\fgood\wedge &b\in f(\delta^X)\}
 \end{align*}
 is stationary in $[H_{\theta}]^\omega$.
 \end{enumerate} 
$\diaba$\index{diamondb0plus@$\diaba$} is the strengthening of $\diab$ where $(ii)$ is replaced by:
\begin{enumerate}
\item[$(ii)^\plus$] For any regular $\theta\geq\omega_2$
$$\{X\prec H_\theta\mid X \text{ is }\fgood\}$$
contains a club of $[H_\theta]^{\omega}$. Moreover, for any $b\in\BP$ 
$$\{\alpha<\omega_1\mid b\in f(\alpha)\}$$
is stationary.
\end{enumerate}
We say that $f$ \textit{witnesses} $\diab$, $\diaba$ respectively. 
\end{defn}

\begin{rem}
Observe that if $f$ witnesses $\diab$ and $\BP$ is separative then $\BP$ can be ``read off" from $f$: We have $\BP=\bigcup_{\alpha<\omega_1} f(\alpha)$ and for $b, c\in\BP$, $b\leq_{\BP} c$ iff whenever $b\in f(\alpha)$ then $c\in f(\alpha)$ as well. Thus, it is usually not necessary to mention $\BP$.
\end{rem}

We introduce some convenient shorthand notation.
\begin{defn}
If $\BP\subseteq\omega_1$ is a forcing, $f$ guesses $\BP$-filters and $b\in\BP$ then
$$S^f_b\coloneqq \{\alpha<\omega_1\mid b\in f(\alpha)\}\index{Sfb@$S^f_b$}.$$
If $f$ is clear from context we will sometimes omit the superscript $f$.
\end{defn}

Note that if $f$ witnesses $\diab$, then $S^f_b$ is stationary for all $b\in\BP$. This is made explicit for $\diaba$. This is exactly the technical strengthening over Woodin's definition of $\diacol, \diacolp$. Lemma \ref{diacomplemm} shows that this strengthening is natural. Moreover, this implies $$\diamondsuit(\BB\oplus\mathbb C)\Rightarrow\diab\wedge\dia{\mathbb C}$$
whenever $\BB, \mathbb C\subseteq\omega_1$ are forcings and $\BB\oplus\mathbb C$ is the disjoint union of $\BB$ and $\mathbb C$ coded into a subset of $\omega_1$. This becomes relevant in Subsection \ref{splitwitnessessubsection}. Nonetheless, the basic theory of these principles is not changed by a lot.

\begin{defn}
If $f$ witnesses $\diab$ and $\PP$ is a forcing, we say that $\PP$ \textit{preserves} $f$\index{Forcing!preserves f@preserves $f$} if whenever $G$ is $\PP$-generic then $f$ witnesses $\diab$ in $V[G]$.
\end{defn}
We remark that if $f$ witnesses $\diaba$ then ``$\PP$ preserves $f$" still only means that $f$ witnesses $\diab$ in $V^\PP$.\medskip

Next, we define a variant of stationary sets related to a witness of $\diab$. Suppose $\theta\geq\omega_2$ is regular. Then $S\subseteq\omega_1$ is stationary iff for any club $\mathcal C\subseteq[H_\theta]^\omega$, there is some $X\in\mathcal C$ with $\delta^X\in S$. $f$-stationarity results from restricting to $f$-slim $X\prec H_\theta$ only.

\begin{defn}
Suppose $f$ guesses $\BP$-filters.
\begin{enumerate}[label=\rn*]
\item A subset $S\subseteq\omega_1$ is $f$\textit{-stationary}\index{f stationary@$f$-stationary!in omega one@in $\omega_1$} iff whenever $\theta\geq\omega_2$ is regular and $\mathcal C\subseteq[H_\theta]^\omega$ is club then there is some $\fgood$ $X\in\mathcal C$ with $\delta^X\in S$.
\item A forcing $\PP$ \textit{preserves} $f$\textit{-stationary sets}\index{Forcing!preserves fstationary sets@preserves $f$-stationary sets} iff any $f$-stationary set is still $f$-stationary in $V^\PP$.
\end{enumerate}
\end{defn}

We make use of $f$-stationarity only when $f$ witnesses $\diab$. However, with the above definition it makes sense to talk about $f$-stationarity in a forcing extension before we know that $f$ has been preserved. Note that all $f$-stationary sets are stationary, but the converse might fail, see Proposition \ref{noplusprop}. We will later see that $f$-stationary sets are the correct replacement of stationary set in our context. Most prominently this notion will be used in the definition of the $\MMpp$-variant $\fMMpp$ we introduce in Subsection \ref{fmppandrelatedforcingaxiomssection}. It will be useful to have an equivalent formulation of $f$-stationarity at hand. 

\begin{prop}\label{fstationaryequivformulationprop}
Suppose $f$ guesses $\BP$-filters. The following are equivalent for any set $S\subseteq\omega_1$:
\begin{enumerate}[label=\rn*]
\item\label{fstationaryequivformulationcond1} $S$ is $f$-stationary.
\item\label{fstationaryequivformulationcond2} Whenever $\langle D_\alpha\mid\alpha<\omega_1\rangle$ is a sequence of dense subsets of $\BP$, the set 
$$\{\alpha\in S\mid \forall\beta<\alpha\ f(\alpha)\cap D_\beta\neq\emptyset\}$$
is stationary.
\end{enumerate}
\end{prop}

\begin{prop}\label{diabequivprop}
Suppose $f$ guesses $\BP$-filters. The following are equivalent:
\begin{enumerate}[label=\rn*]
\item\label{diabequivcond1} $f$ witnesses $\diab$.
\item\label{diabequivcond2} $S^f_b$ is $f$-stationary for all $b\in\BP$.
\item\label{diabequivcond3} For any $b\in\BP$ and sequence $\langle D_\alpha\mid\alpha<\omega_1\rangle$ of dense subsets of $\BP$, 
$$\{\alpha\in S^f_b\mid\forall\beta<\alpha\ f(\alpha)\cap D_\beta\neq\emptyset\}$$
is stationary.
\end{enumerate}
\end{prop}

We mention a handy corollary.

\begin{cor}\label{preservefcor}
Suppose $f$ witnesses $\diab$. Any forcing preserving $f$-stationary sets preserves $f$.
\end{cor}

\begin{prop}\label{diabaequivprop}
Suppose $f$ guesses $\BP$-filters. The following are equivalent:
\begin{enumerate}[label=\rn*]
\item\label{diabaequivcond1} $f$ witnesses $\diaba$.
\item\label{diabaequivcond2} For any $b\in\BP$, $S^f_b$ is stationary and all stationary sets are $f$-stationary.
\item\label{diabaequivcond3} If $D$ is dense in $\BP$ then 
$$\{\alpha<\omega_1\mid\ f(\alpha)\cap D\neq\emptyset\}$$
contains a club and for all $b\in\BP$, $S^f_b$ is stationary.
\item\label{diabaequivcond4} All countable $X\prec H_\theta$ with $f\in X$ and $\theta\geq \omega_2$ regular are $\fgood$ and moreover for all $b\in \BP$, $S^f_b$ is stationary.
\end{enumerate}
\end{prop}

We will now give a natural equivalent formulation of $\diaba$.

Witnesses of $\diaba$ are simply codes for regular embeddings\footnote{Regular embeddings, also known as complete embeddings, are embeddings between partial orders which preserve maximal antichains.} of $\BP$ \mbox{into $\NS^+$}.

\begin{lemm}\label{diacomplemm}
The following are equivalent:
\begin{enumerate}[label=\rn*]
\item\label{diacompcond1} $\diaba$.
\item\label{diacompcond2} There is a regular embedding $\eta\colon\BP\rightarrow(\Pomo/\NS)^+$.
\end{enumerate}
\end{lemm}

The argument above suggests the following definition.
\begin{defn}
Suppose $f$ witnesses $\diab$. We define 
$$\eta_f\colon \BP\rightarrow (\Pomo/\NS)^+$$
by $b\mapsto[S^f_b]_{\NS}$ and call $\eta_f$ the \textit{embedding associated to} $f$\index{etaf@$\eta_f$ (the embedding associated to $f$)}.
\end{defn}

We will now show that $\diab$ is consistent for any forcing $\BP\subseteq\omega_1$, even simultaneously so for all such $\BP$. We will deal with the consistency of $\diaba$ in the next section.

\begin{prop}\label{diamondprop}
Assume $\diamondsuit$. Then $\diab$ holds for any poset $\BP\subseteq\omega_1$.
\end{prop}

\begin{cor}\label{consistentcor}
Suppose $\BP\subseteq\omega_1$ is a forcing. Then $\diab$ holds in $V^{\Add(\omega_1,1)}$.
\end{cor}

In a number of arguments, we will deal with $\fgood$ $X\prec H_\theta$ that become thicker over time, i.e.~at a later stage there will be some $\fgood$ $X\sqsubseteq Y\prec H_\theta$. 

\begin{defn}\label{canonicalembeddingdefn}
In the above case of $X\sqsubseteq Y$, we denote the canonical elementary embedding from $M_X$ to $M_Y$ by 
$$\mu_{X, Y}\colon M_X\rightarrow M_Y.$$
$\mu_{X, Y}$\index{muXY@$\mu_{X, Y}$} is given by $\pi_Y^{-1}\circ\pi_X$.
\end{defn}

Usually, both $X$ and $Y$ will be $\fgood$. It is then possible to lift $\mu_{X, Y}$.

\begin{prop}\label{canonicalliftprop}
Suppose $f$ guesses $\BP$-filters and $X, Y\prec H_\theta$ are both $\fgood$ with $X\sqsubseteq Y$. Then the lift of $\mu_{X, Y}$ to 
$$\mu_{X, Y}^+\colon M_X[f(\delta^X)]\rightarrow M_Y[f(\delta^X)]$$
exists.
\end{prop}

\begin{proof}
As $\delta^X=\delta^Y$, the critical point of $\mu_{X, Y}$ is ${>}\delta^X$ (if it exists). As $\pi_X^{-1}(\BP)$ is a forcing of size ${\leq}\omega_1^{M_X}=\delta^X$ and $f(\delta^X)$ is generic over both $M_X$ and $M_Y$, the lift exists.
\end{proof}

We consider the above proposition simultaneously as a definition: From now on $\mu_{X, Y}^+$\index{muXYplus@$\mu_{X, Y}^+$} will refer to this lift if it exists.

\begin{defn}
Suppose $f$ witnesses $\diab$. $\NSf$\index{NSf@$\NSf$ (the ideal of $f$-nonstationary sets)} is the ideal of $f$-nonstationary sets, that is 
$$\NSf=\{N\subseteq\omega_1\mid N\text{ is not }f\text{-stationary}\}.$$
\end{defn}

\begin{lemm}\label{nsfnormalideallemm}
Suppose $f$ witnesses $\diab$. $\NSf$ is a normal uniform ideal.
\end{lemm}

\subsection{Miyamoto's theory of nice iterations}

For all our intents and purposes, it does not matter in applications how the limit our iterations look like as long as we can prove a preservation theorem about it.\medskip

We give a brief introduction to Miyamoto's theory of nice iterations. These iterations are an alternative to $\mathrm{RCS}$-itertaions when dealing with the problem described above. In the proof of the iteration theorem for ($f$-)proper forcings, one constructs a generic condition $q$ by induction as the limit of a sequence $\langle q_n\mid n<\omega\rangle$. In case of \mbox{($f$-)}semiproper forcings, the length of the iteration may have uncountable cofinality in $V$ but become $\omega$-cofinal along the way. In this case, a sequence $\langle q_n\mid n<\omega\rangle$ with the desired properties cannot be in $V$. The key insight to avoid this issue is that one should give up linearity of this sequence and instead build a tree of conditions in the argument. Nice supports follow the philosophy of form follows function, i.e.~its definitions takes the shape of the kind of arguments it is intended to be involved in. The conditions allowed in a nice limit are represented by essentially the kind of trees that this inductive nonlinear constructions we hinted at above produces.\\

Miyamoto works with a general notion of iteration. For our purposes, we will simply define nice iterations by induction on the length. Successor steps are defined as usual, that is if $\PP_\gamma=\langle \PP_\alpha, \Q_\beta\mid \alpha\leq\gamma, \beta<\gamma\rangle$ is a nice iteration of length $\gamma$ and $\Q_\gamma$ is a $\PP_\gamma$-name for a forcing then $\langle \PP_\alpha, \Q_\beta\mid \alpha\leq\gamma+1, \beta\leq \gamma\rangle$ is a nice iteration of length $\gamma+1$ where $\PP_{\gamma+1}\cong \PP_\gamma\ast\Q_\gamma$. 

\begin{defn}[Miyamoto, \cite{miyamotosemiproper}]
Let $\vec \PP=\langle \PP_\alpha,\Q_\alpha\mid\alpha<\gamma\rangle$ be a potential nice iteration, that is

\begin{enumerate}[label=$(\vec{\PP}.\roman*)$]
    \item $\PP_\alpha$ is a nice iteration of length $\alpha$ for all $\alpha<\gamma$,
    \item $\PP_{\alpha+1}\cong\PP_\alpha\ast\Q_\alpha$ for all $\alpha+1<\gamma$ and
    \item $\PP_\beta\res\alpha=\PP_\alpha$ for all $\alpha\leq\beta<\gamma$.
\end{enumerate}
A \textit{nested antichain}\index{Nested antichain} in $\vec\PP$ is of the form 
$$(T, \langle T_n\mid n<\omega\rangle, \langle \sucm^n_T\mid n<\omega\rangle)$$
so that for all $n<\omega$ the following hold\footnote{Usually, we identify the nested antichain with $T$, its first component and write $\sucm(a)$ instead of $\sucm^n_T(a)$ if $n, T$ are clear from context.}:
\begin{enumerate}[label=$(\roman*)$]
\item $T=\bigcup_{n<\omega} T_n$.
\item $T_0=\{a_0\}$ for some $a_0\in\bigcup_{\alpha<\gamma}\PP_\alpha$.
\item $T_n\subseteq\bigcup_{\alpha<\gamma}\PP_\alpha$ and $\sucm^n_T\colon T_n\rightarrow \mathcal P(T_{n+1})$.
\item For $a\in T_n$ and $b\in\sucm^n_T(a)$, $\lh(a)\leq\lh(b)$ and $b\res\lh(a)\leq a$.
\item For $a\in T_n$ and distinct $b, b'\in \sucm^n_T(a)$, $b\res \lh(a)\perp b'\res\lh(a)$.
\item For $a\in T_n$, $\{b\res\lh(a)\mid b\in\sucm^n_T(a)\}$ is a maximal antichain below $a$ in $\PP_{\lh(a)}$.
\item $T_{n+1}=\bigcup\{\sucm^n_T(a)\mid a\in T_n\}$.
\end{enumerate}
Abusing notation, we will usually identify\index{sucnT@$\sucm^n_T(a)$} $T$ with 
$$(T, \langle T_n\mid n<\omega\rangle, \langle \sucm^n_T\mid n<\omega\rangle).$$
If $b\in \sucm^n_T(a)$ then we also write $a=\pred^n_T(b)$\index{prednT@$\pred^n_T(b)$}.
If $\beta<\gamma$ then $p\in \PP_\beta$ is a \textit{mixture of}\index{Mixture of $T$ up to $\beta$} $T$ \textit{up to} $\beta$ iff for all $\alpha<\beta,\ p\res\alpha$ forces
\begin{enumerate}[label=$(p.\roman*)$]
\item $p(\alpha)=a_0(\alpha)$ if $\alpha<\lh(a_0)$ and $a_0\res \alpha\in G_{\alpha}$,
\item $p(\alpha)=b(\alpha)$ if there are $a, b\in T$, $n<\omega$ with $b\in\sucm^n_T(a)$, $\lh(a)\leq\alpha<\lh(b)$ and $b\res\alpha\in G_\alpha$,
\item $p(\alpha)=\1_{\Q_\alpha}$ if there is a sequence $\langle a_n\mid n<\omega\rangle$ with $a_{n+1}\in\sucm^n_T(a_n)$, $\lh(a_n)\leq\alpha$ and $a_n\in G_{\lh(a_n)}$ for all $n<\omega$.
\end{enumerate}
If $\xi\leq\gamma$ is a limit, and $q$ is a sequence of length $\xi$ (may or may not be in $\PP_\xi$), $q$ is $(T, \xi)$\textit{-nice} if for all $\beta<\xi$, $q\res\beta\in\PP_\beta$ is a mixture of $T$ up to $\beta$.
\end{defn}

We refer to \cite{miyamotosemiproper} for basic results on nested antichains and mixtures. We go on and define nice limits.

\begin{defn}[Miyamoto, \cite{miyamotosemiproper}]
Suppose $\vec\PP=\langle \PP_\alpha,\Q_\alpha\mid\alpha<\gamma\rangle$ is a potential nice iteration of limit length $\gamma$. Let $\bar \PP$ denote the inverse limit along $\vec \PP$. The \textit{nice limit}\index{Nice limit} of $\vec \PP$ is defined as
$$\nicelim(\vec\PP)=\{p\in\bar\PP\mid\exists T\text{ a nested antichain of }\vec\PP\text{ and }p\text{ is }(T, \gamma)\text{-nice}\}.$$
$\nicelim(\vec\PP)$ inherits the order from $\bar\PP$.
\end{defn}
Finally, if $\vec \PP=\langle \PP_\alpha, \Q_\alpha\mid\alpha<\gamma\rangle$ is a potential nice iteration then
$$\langle \PP_\alpha,\Q_\beta\mid\alpha\leq\gamma,\beta<\gamma\rangle$$
is a \textit{nice iteration} \index{Nice iteration}of length $\gamma$ where $\PP_\gamma=\nicelim(\vec \PP)$.\\

The fundamental property of nice iterations is:
\begin{fact}[Miyamoto,\cite{miyamotosemiproper}]\label{m3.2fact}
Suppose $\PP=\standarditeration$ is a nice iteration and $T$ is a nested antichain in $\PP$. Then there is a mixture \mbox{of $T$}.
\end{fact}

\begin{defn}[Miyamoto,\cite{miyamotosemiproper}]
Let $\PP=\standarditeration$ be a nice iteration. If $S, T$ are nested antichains in $\PP$ then $S\hooks T$\index{<hooks@$\hooks$} iff for any $n<\omega$ and $a\in S_n$ there is $b\in T_{n+1}$ with 
$$\lh(b)\leq \lh(a)\text{ and }a\res\lh(b)\leq b.$$
\end{defn}

\begin{fact}[Miyamoto, \cite{miyamotosemiproper}]\label{m2.11fact} Let $\PP=\langle \PP_\alpha,\Q_\beta\mid\alpha\leq\gamma, \beta<\gamma\rangle$ be a nice iteration of limit length $\gamma$. Suppose that
\begin{enumerate}[label=\rn*]
\item $T$ is a nested antichain in $\PP$,
\item $p$ is a mixture of $T$ and $s\in\PP$,
\item $r\in T_1$,
\item $s\leq r^\frown p\res[\lh(r), \gamma)$ and
\item $A\subseteq\gamma$ is cofinal.
\end{enumerate}
Then there is a nested antichain $S$ in $\PP$ with 
\begin{enumerate}[label=$(\alph*)$]
\item $s$ is a mixture of $S$,
\item If $S_0=\{c\}$ then $\lh(r)\leq\lh(c)\in A$ and $c\res\lh(r)\leq r$ and
\item $S\hooks T$.
\end{enumerate}
\end{fact}

The following describes the tool we use to construct conditions. 

\begin{defn}[Miyamoto, \cite{miyamotosemiproper}]
Let $\PP=\standarditeration$ be a nice iteration of limit length $\gamma$. A \textit{fusion structure}\index{Fusion structure} in $\PP$ is 
$$T, \langle \msup{p}{a, n},\ \msup{T}{a, n}\mid n<\omega,\ a\in T_n\rangle$$
where 
\begin{enumerate}[label=\rn*]
\item $T$ is a nested antichain in $\PP$
\end{enumerate} 
and for all $n<\omega$ and $a\in T_n$
\begin{enumerate}[label=$(\roman*)$, resume]
\item $\msup{T}{a, n}$ is a nested antichain in $\PP$,
\item $\msup{p}{a, n}\in\PP$ is a mixture of $\msup{T}{a, n}$,
\item $a\leq \msup{p}{a, n}\res\lh(a)$ and if $\msup{T_0}{a, n}=\{c\}$ then $\lh(a)=\lh(c)$ and
\item for any $b\in\sucm_T^n(a)$, $\msup{T}{b, n+1}\hooks \msup{T}{a, n}$, thus $\msup{p}{b, n+1}\leq \msup{p}{a, n}$.
\end{enumerate}
If $q\in\PP$ is a mixture of $T$ then $q$ is called a \textit{fusion}\index{Fusion structure!fusion of} of the fusion structure.
\end{defn}

\begin{fact}[Miyamoto, \cite{miyamotosemiproper}]\label{m3.5fact} Let $\PP=\standarditeration$ be a nice iteration of limit length $\gamma$. If $q\in\PP$ is a fusion of a fusion structure 
$$T, \langle \msup{p}{a, n},\ \msup{T}{a, n}\mid n<\omega,\ a\in T_n\rangle$$
and $G$ is $\PP$-generic with $q\in G$ then the following holds in $V[G]$:
There is a sequence $\langle a_n\mid n<\omega\rangle$ so that for all $n<\omega$
\begin{enumerate}[label=\rn*]
\item $a_0\in T_0$,
\item $a_n\in G_{\lh(a_n)}$,
\item $a_{n+1}\in\sucm_T^{n}(a_n)$ and
\item $\msup{p}{a_n, n}\in G$.
\end{enumerate}
\end{fact}

We mention one more convenient fact:

\begin{fact}[Miyamoto, \cite{miyamotosimple}]\label{niceccfact}
Suppose $\kappa$ is an inaccessible cardinal, $\PP=\langle \PP_\alpha,\Q_\beta\mid\alpha\leq\kappa,\beta<\kappa\rangle$ is a nice iteration so that
\begin{enumerate}[label=$(\roman*)$]
\item $\vert\PP_\alpha\vert<\kappa$ for all $\alpha<\kappa$ and
\item $\PP$ preserves $\omega_1$.
\end{enumerate}
Then $\PP$ is $\kappa$-c.c..
\end{fact}
Miyamoto proves this for so called simple iterations of semiproper forcings. The proof works just as well for nice iterations of semiproper forcings and finally the proof can be made to work with assuming only $\PP$ preserves $\omega_1$ instead of $\PP$ being a semiproper iteration.

\section{The Iteration Theorem}\label{qmmsection}

The full main theorem we are going to prove is the following.
\begin{thm}\label{iterationtheorem}
    Suppose $f$ witnesses $\diab$ and $\PP=\standarditeration$ is a nice iteration of $f$-preserving forcings. Suppose that 
    \begin{enumerate}[label=$(\PP.\roman*)$]
        \item\label{itm:srprestriction} $\forces_{\PP_{\alpha+2}}\SRP$ for all $\alpha+2\leq\gamma$ and
        \item\label{itm:harshrestriction} $\forces_{\PP_\alpha}``\Q_{\alpha}\text{ preserves }f\text{-stationary sets from } \bigcup_{\beta<\alpha}V[\dot G_\beta]$.
    \end{enumerate}
    Then $\PP$ preserves $f$.
\end{thm}
Note that if $\BB$ is the trivial forcing $\{\1\}$ and we take $f$ to be the witness of $\diaba$ with $f(\alpha)=\{\1\}$ for all $0<\alpha<\omega_1$, then we recover the special case mentioned in the introduction.

So what is the basic idea? For the moment, let us assume that $f$ is the trivial witness of $\diamondsuit(\{\1\})$ above for simplicity. As always, we want to imitate the argument of the mother of all iteration theorems, the iteration theorem for proper forcings.
Suppose we have a full support iteration 
$$\PP=\langle \PP_n, \Q_m\mid n\leq\omega, m<\omega\rangle$$
and for the moment assume only that 
$$\forces_{\PP_n}``\Q_n\text{ preserves }\omega_1".$$
 We try to motivate some additional reasonable constraints imply $\PP$ to be $\omega_1$-preserving. For the moment, we try to consider Shelah's argument as a game: In the beginning there some countable $X\prec H_\theta$ as well as $p_0\in X\cap \PP$.
The argument proceeds as follows: In round $n$, we have already constructed a $(X, \PP_{n})$-semigeneric condition $q_n\in\PP\res n$ and have 
$$q_n\forces \dot p_{n}\res n\in\dot G_{n}\cap \check X[\dot G_n].$$
Next, our adversary hits us with a dense subset $D\subseteq\PP$ in $X$ and we must find  $\dot p_{n+1}\in V^{\PP_n}$ and some $(X, \PP_{n+1})$-semigeneric $q_{n+1}$ with $q_{n+1}\res n=q_n$ and\footnote{Here, we consider $\dot p_n$ also as a $\PP_{n+1}$-name.}
$$q_{n+1}\forces \dot p_{n+1}\in\check D\wedge p_{n+1}\res n+1\in \dot G_{n+1}\cap \check X[\dot G_{n+1}].$$
 Our job is to survive this game for $\omega$-many steps. If we have a winning strategy then we can find a $(X, \PP)$-semigeneric condition, so in particular $\PP$ preserves $\omega_1$.

Destroying stationarity makes it significantly more difficult to survive the above game: Suppose for example that 
$$p_0(0)\forces\check S\in\NS$$
for some $S\in X$ with $\delta^X\in S$. Then there is no hope of finding a $(X, \PP_{1})$-semigeneric condition $q$ with $q\leq p_0\res 1$. 
Hence, we must already be careful with what $X$ we start the game. This leads us to the following definitions.

\begin{defn}
Suppose $\theta$ is sufficiently large and regular, $X\prec H_\theta$ is countable. If $I$ is an ideal on $\omega_1$, we say that $X$ \textit{respects} $I$\index{X respects I@$X$ respects $I$} if for all $A\in I\cap X$ we have $\delta^X\notin A$.
\end{defn}

Note that all countable $X\prec H_\theta$ respect $\NS$ and countable $Y\prec H_\theta$ with $f\in Y$ respects $\NSf$ if and only if $Y$ is $\fgood$.

\begin{defn}
Suppose $\PP$ is a forcing and $\dot I\in V^\PP$ is a name for an ideal on $\omega_1$. For $p$ in $\PP$, we denote the \textit{partial evaluation of} $\dot I$ \textit{by} $p$\index{Partial Evaluation} by
$$\dot I^p\coloneqq \{S\subseteq\omega_1\mid p\forces\check S\in\dot I\}.$$
\end{defn}

Back to the discussion, we need to start with an $X$ so that $X$ respects $\dot I^{p_0\res 1}$ where $\dot I$ is a name for the nonstationary ideal. This gives us a shot at getting past the first round. Luckily, there are enough of these $X$.

\begin{defn}
Let $A$ be an uncountable set with $\omega_1\subseteq A$ and $I$ a normal uniform ideal on $\omega_1$. Then $\mathcal S\subseteq [A]^\omega$ is \textit{projective} $I$\textit{-positive}\index{[A]omega@$[A]^\omega$!Projective I positive@projective $I$-positive in} if for any $S\in I^+$ the set 
$$\{X\in\mathcal S\mid \delta^X\in S\}$$
is stationary in $[A]^\omega$.
\end{defn}

\begin{prop}\label{trueprojectiveprop}
Suppose $\theta$ is sufficiently large and regular. Let $I$ be a normal uniform ideal on $\omega_1$. Then 
$$\mathcal S=\{X\in[H_\theta]^\omega\mid X\prec H_\theta\text{ respects }I\}$$
is projective $I$-positive.
\end{prop}

\begin{proof}
Let $\mathcal C$ be a club in $[H_\theta]^\omega$ and assume that all elements of $\mathcal C$ are elementary substructures of $H_\theta$ and contain $I$ as an element. Let 
$$\vec X\coloneqq \langle X_\alpha\mid\alpha<\omega_1\rangle$$
be a continuous increasing chain of elements in $\mathcal C$. Let $X\coloneqq\bigcup_{\alpha<\omega_1} X_\alpha$ and let 
$$\vec A\coloneqq \langle A_\alpha\mid\alpha<\omega_1\rangle$$
be an enumeration of $X\cap I$. Let $C\subseteq\omega_1$ be the set of $\alpha$ so that 
\begin{enumerate}[label=$(C.\roman*)$]
\item $\delta^{X_\alpha}=\alpha$ and
\item $\vec A\res\alpha$ is an enumeration of $X_\alpha\cap I$
\end{enumerate}
and note that $C$ is club. Let $A=\bigtriangledown_{\alpha<\omega_1} I_\alpha$. As $I$ is normal, $A\in I$. Then $C- A$ is a complement of a set in $I$ and for any $\alpha\in C-A$ we have
$$\delta^{X_\alpha}=\alpha\notin I_\beta$$
for all $\beta<\alpha$. Hence $X_\alpha\in\mathcal S\cap\mathcal C$.
\end{proof}

 Of course, the problem continues. What if we have found a suitable $q_1$ and now we work in $V[G_{1}]$ with $q_1\in G_{1}$. At the very least, we need that $X[G_{1}]$ respects $\dot I^{p_0\res [1,2)}$, where $\dot I$ is now a $\PP_{1, 2}$-name for the nonstationary ideal. Ensuring this is a matter of being able to pick the right $q_1$ to begin with. This motivates the following class of forcings.

\begin{defn}\label{namedichotomydef}
We say that a forcing $\PP$ is \textit{respectful}\index{Forcing!respectful} if $\PP$ preserves $\omega_1$ and the following is true: Whenever
\begin{itemize}
 \item $\theta$ is sufficiently large and regular,
 \item $X\prec H_\theta$ is countable with $\PP\in X$,
 \item $\dot I\in X$ is a $\PP$-name for a normal uniform ideal and
 \item $p\in\PP\cap X$
 \end{itemize}  
 then exactly one of the following holds:
 
\begin{enumerate}[label=$(\mathrm{Res}.\roman*)$]
\item\label{namedichotomyopt1} Either there is some $(X, \PP)$-semigeneric $q\leq p$ so that 
$$q\forces`` \check X[\dot G] \text{ respects }\dot I"$$
or
\item\label{namedichotomyopt2} $X$ does not respect $\dot I^p$.
\end{enumerate}
\end{defn}

Roughly, this condition states that we can find a $\PP$-generic filter $G$ with $p\in G$ so that $X\sqsubseteq X[G]$ respects $\dot I^G$ as long as there is no obvious obstruction to it.

\begin{rem}
If $\PP$ is respectful and preserves stationary sets then $\PP$ is semiproper. However, the converse is not true in general. Similarly, a respectful $f$-stationary set preserving forcing is $f$-semiproper, which follows from plugging in a name for $\NSf$ as $\dot I$ in the definition of respectfulness.
\end{rem}

We require\footnote{This excludes the first counterexample due to Shelah, but not yet all the counterexamples of the second kind.} now that
$$\forces_{\PP_n}``\Q_n\text{ is respectful}"$$
for all $n<\omega$. We then aim to make sure (assuming $\dot p_{n+1}$ is already defined) to find $q_{n+1}$ in round $n$ so that in addition to the prior constraints,
$$q_{n+1}\forces ``\check X[\dot G_{n+1}]\text{ respects } \dot I"$$
where $\dot I$ is a $\PP_{n+1}$ name for the ideal of sets forced to be nonstationary by $\dot p_{n+1}(n+1)$. Consider $\dot I$ as a $\PP_n$-name $\ddot I$ for a $\Q_n$-name. By respectfulness, this reduces to avoiding an instance of the ``bad case" \ref{namedichotomyopt2}, namely we should make sure that whenever $G_n$ is $\PP_n$-generic with $q_n\in G_n$ then 
$$X[G_n]\text{ respects }\left(\ddot I^{G_n}\right)^{p_{n+1}(n+1)}$$
where $p_{n+1}=\dot p_{n+1}^{G_{n+1}}$.
he next key insight is that this reduces to 
$$``X[G_n]\text{ respects }J\coloneqq\{S\subseteq\omega_1\mid p_{n+1}(n)\forces \check S\in\NS\}"$$
which we have (almost)\footnote{We made sure of this if $p_{n+1}$ is replaced by $p_n$ in the definition of $J$, we ignore this small issue for now.} already justified inductively, assuming $\Q_{n+1}$ only kills new stationary sets: Our final requirement\footnote{It is readily seen that this eliminates the counterexamples of the second kind.} is that 
$$\forces_{\PP_{n+1}}``\Q_{n+1}\text{ preserves stationary sets which are in }V[\dot G_n]"$$
for all $n<\omega$. The point is that trivially $\left(\ddot I^{G_n}\right)^{p_{n+1}(n)}$ only contains sets in $V[G_n]$, so all such sets will be preserved by $\Q_{n+1}$. The sets that are killed are then already killed in the extension by $\Q_n^{G_n}$.\\
Modulo some details we have shown the following.
\begin{thm}\label{omegakilliterationthm}
Suppose $\PP=\langle \PP_n, \Q_m\mid n\leq\omega, m<\omega\rangle$ is a full support iteration so that
\begin{enumerate}[label=$(\PP.\roman*)$]
    \item $\forces_{\PP_n}\Q_n\text{ is respectful}$ and
    \item\label{omegakilliterationcond2} $\forces_{\PP_{n+1}}\Q_{n+1}\text{ preserves stationary sets which are in }V[\dot G_n]"$
\end{enumerate}
for all $n<\omega$. Then $\PP$ does not collapse $\omega_1$.
\end{thm}

Two issues arise when generalizing this to longer iterations. The first issue is the old problem that new relevant indices may appear along the iteration in the argument, which we deal with by using nice supports. The second problem is that it seemingly no longer suffices that each iterand individually is respectful. For longer iterations, say of length $\gamma$, the argument then requires that 
$$\forces_{\alpha}``\dot\PP_{\alpha,\beta}\text{ is respectful}"$$
for sufficiently many $\alpha<\beta<\gamma$. This is problematic as we will not prove an iteration theorem of any kind for respectful forcings\footnote{Indeed it seems that no useful iteration theorem for respectful forcings is provable in $\ZFC$, see Subsection \ref{disrespectfulforcingsubsection}.}. This is where we take out the sledgehammer.
\begin{defn}
$(\ddagger)$\index{(double dagger)@$(\ddagger)$} holds if and only if all $\omega_1$-preserving forcings are respectful.
\end{defn}
\begin{lemm}\label{sledgehammerlemm}
$\SRP$ implies $(\ddagger)$.
\end{lemm}

\begin{proof}
Let $\PP$, $\theta$, $\dot I$, $p$  be as Definition \ref{namedichotomydef}. It is easy to see that \ref{namedichotomyopt1} and \ref{namedichotomyopt2} cannot hold simultaneously. It is thus enough to prove that one of them holds.  Let $\lambda$ be regular, $2^{\vert\PP\vert}<\lambda<\theta$ and $\lambda\in X$ and consider the set 
\begin{align*}
\mathcal S=\{Y\in [H_\lambda]^\omega\mid& Y\prec H_\lambda\wedge \neg(\exists q\leq p\ q\text{ is }\\
&(Y, \PP)\text{-semigeneric}\text{ and }q\forces``\check Y[\dot G]\text{ respects }\dot I")\}.
\end{align*}

By $\SRP$, there is a continuous increasing elementary chain 
$$\vec Y=\langle Y_\alpha\mid\alpha<\omega_1\rangle$$
so that
\begin{enumerate}[label=$(\vec Y.\roman*)$]
\item $\PP, p, \dot I\in Y_0$ and
\item for all $\alpha<\omega_1$, either $Y_\alpha\in\mathcal S$ or there is no $Y_\alpha\sqsubseteq Z\prec H_\theta$ with $Z\in\mathcal S$.
\end{enumerate}
Let $S=\{\alpha<\omega_1\mid Y_\alpha\in\mathcal S\}$.

\begin{claim}
$p\forces\check S\in\dot I$.
\end{claim}

\begin{proof}
Let $G$ be generic with $p\in G$ and let $I=\dot I^G$.   Assume toward a contradiction that $S$ is $I$-positive.  Note that $\langle Y_\alpha[G]\mid\alpha<\omega_1\rangle$ is a continuous increasing sequence of elementary substructure of $H_\theta^{V[G]}$. Hence there is a club $C$ of $\alpha$ so that for $\alpha\in C$
$$\delta^{Y_\alpha}=\delta^{Y_\alpha[G]}=\alpha$$
and thus there is a $(Y_\alpha, \PP)$-semigeneric condition $q\leq p$, $q\in G$. Hence by definition of $S$, for any $\alpha\in S\cap C$, we may find some $N_\alpha\in I\cap Y_\alpha[G]$ so that $\delta^{Y_{\alpha}}\in N_\alpha$. By normality of $I$, there is some $I$-positive $T\subseteq S\cap C$ and some $N$ so that $N=N_\alpha$ for all $\alpha\in T$. But then for $\alpha\in T$, we have 
$$\alpha=\delta^{Y}\in N$$
so that $T\subseteq N$. But $N\in I$, contradiction.
\end{proof}
Thus if $\delta^X\in S$, then $S$ witnesses \ref{namedichotomyopt2} to hold. Otherwise, $\delta^X\notin S$. Note that $\delta^{Y_{\delta^X}}=\delta^X$ as $\vec Y\in X$. We find that $Y_{\delta^X}\sqsubseteq X\cap H_\lambda\prec H_\lambda$. Thus, $X\cap H_\lambda\notin\mathcal S$, so that there must be some $q\leq p$ that is $(X\cap H_\lambda, \PP)$-semigeneric and 
$$q\forces``(\widecheck{X\cap H_\lambda})[\dot G]\text{ respects }\dot I".$$
This $q$ witnesses that \ref{namedichotomyopt1} holds.
\end{proof}

We will get around this second issue by forcing $\SRP$ often along the iteration. 
Remember that what we really care about is preserving a witness $f$ of $\diab$ along an iteration of $f$-preserving forcings, so fix such an $f$ now.
It will be quite convenient to introduce some short hand notation.
\begin{defn}
Suppose $\PP$ is a forcing and $p\in \PP$. Then we let $I^\PP_p$\index{IPp@$I^{\PP}_p$} denote $\dot I^p$ where $\dot I$ is a $\PP$-name for $\NSf$. That is
$$I^\PP_p\coloneqq \{S\subseteq\omega_1\mid p\forces\check S\in\NSf\}.$$
\end{defn}

\begin{defn}
Suppose $f$ witnesses $\diab$. An $f$\textit{-ideal}\index{f ideal@$f$-ideal} is an ideal $I$ on $\omega_1$ so that

\begin{enumerate}[label=$(\roman*)$]
\item whenever $S\in I^+$ and $\langle D_i\mid i<\omega_1\rangle$ is a sequence of dense subsets of $\BB$, then 
$$\{\alpha\in S\mid \forall\beta<\alpha\ f(\alpha)\cap D_\beta\neq\emptyset\}\in I^+$$
\item and $S^f_b\in I^+$ for all $b\in \BP$.
\end{enumerate}
\end{defn}

Recall that $\NSf$ is clearly an $f$-ideal and it is normal and uniform by Lemma \ref{nsfnormalideallemm}.

\begin{prop}\label{fidealprop}
Suppose $\PP$ is a forcing that preserves $f$ and $p\in\PP$. Then $I^\PP_p$
is a normal uniform $f$-ideal.
\end{prop}

We leave the proof to the reader. The next Lemma gives us a criterion that guarantees the relevant witness $f$ of $\diab$ to be preserved. We first introduce the notion of a $f$-semigeneric condition.

\begin{defn}
    Suppose $f$ witnesses $\diab$, $\PP$ is a forcing, $\theta$ is sufficiently large and $X\prec H_\theta$ is a $\fgood$ elementary substructure of $H_\theta$ with $\PP\in X$. A condition $p\in \PP$ is called $(X, \PP, f)$-semigeneric if
    $p$ is $(X, \PP)$-semigeneric and 
    $$p\forces \check X[\dot G]\text{ is }\check f\text{-slim}.$$
\end{defn}

\begin{lemm}\label{fpreservinglemm}
Suppose $f$ witnesses $\diab$ and $\PP$ is a forcing with the following property: For any sufficiently large regular $\theta$ and $p\in\PP$ there is a normal uniform $f$-ideal $I$ so that 
$$\{X\in[H_\theta]^\omega\mid X\prec H_\theta\wedge\PP, p\in X\wedge\exists q\leq p\ q\text{ is }(X, \PP, f)\text{-semigeneric}\}$$
is projective $I$-positive. Then $\PP$ preserves $f$.
\end{lemm}

\begin{proof}
Assume $p\in\PP$, $\theta$ is sufficiently large and regular. Let $b\in\BB$,
$$\vec{\dot D}=\langle \dot D_\alpha\mid\alpha<\omega_1\rangle$$
be a sequence of $\PP$-names for dense subsets of $\BB$ and $\dot C$ a $\PP$-name for a club in $\omega_1$. We will find $q\leq p$ so that
$$q\forces\exists\alpha\in S^{\check f}_b\cap\dot C\forall\beta<\alpha\ \check f(\alpha)\cap\dot D_\beta\neq\emptyset.$$
By our assumption, there is some normal uniform $f$-ideal $I$ so that 
$$\{X\in[H_\theta]^\omega\mid X\prec H_\theta\wedge\PP, p\in X\wedge\exists q\leq p\ q\text{ is }(X, \PP, f)\text{-semigeneric}\}$$
is projective $I$-positive. It follows that we can find some countable $X\prec H_\theta$ so that 

\begin{enumerate}[label=$(X.\roman*)$]
\item\label{Xpropertiescond1} $\PP, p, \vec {\dot D}, \dot C\in X$ as well as
\item\label{Xpropertiescond2} $b\in f(\delta^X)$
\end{enumerate}
and some $q\leq p$ that is $(X, \PP, f)$-semigeneric. If $G$ is then any $\PP$-generic with $q\in G$, we have 
$$X\sqsubseteq X[G]\text{ is }\fgood$$
and hence $\delta^X\in\dot C^G$ as well as 
$$\forall\beta<\delta^X\ f(\delta^X)\cap \dot D^G_\beta\neq\emptyset.$$
\end{proof}

We also need to resolve a small issue that we glossed over in the sketch of a proof of Theorem \ref{omegakilliterationthm}.

\begin{lemm}\label{technicalQlemm}
Suppose $f$ witnesses $\diab$. Further assume that 
\begin{itemize}
\item $\PP$ is a respectful, $f$-preserving forcing and $p\in\PP$,
\item $\theta$ is sufficiently large and regular,
\item $X\prec H_\theta$ is countable, respects $I^\PP_p$ and $\PP, p\in X$ and
\item $M_X[f(\delta^X)]\models``D\text{ is dense below }\pi_X^{-1}(p)\text{ in }\pi_X^{-1}(\PP)"$.
 \end{itemize}
 Then there are $Y, q$ with
\begin{enumerate}[label=$(\roman*)$]
\item $X\sqsubseteq Y\prec H_\theta$ is countable,
\item $q\leq p$,
\item $Y$ respects $I^\PP_q$, in particular $Y$ is $\fgood$ and
\item $q\in \pi_Y[\mu_{X, Y}^+(D)]$.
\end{enumerate}
\end{lemm}

\begin{proof}
We may assume that $X$ is an elementary substructure of 
$$\mathcal H\coloneqq (H_\theta;\in, \unlhd)$$
where $\unlhd$ is a wellorder of $H_\theta$. As $\PP$ is respectful and $X$ respects $I^\PP_p$, there is a $(X, \PP)$-semigeneric condition $r\leq p$ so that 
$$r\forces`` \check X[\dot G]\text{ respects }\NScheckf"$$
i.e.~$r$ is $(X, \PP, f)$-semigeneric. Let $G$ be $\PP$-generic with $r\in G$. Then $X[G]$ is $\fgood$. Let $Z=X[G]\cap V$, note that $\mu_{X, Z}^+$ exists by Proposition \ref{canonicalliftprop}. Now there is thus some $q\leq p$, $q\in G$ with 
$$q\in\pi_Z[\mu_{X, Z}^+(D)].$$
Finally, note that $q$ and $Y\coloneqq \mathrm{Hull}^\mathcal H(X\cup\{q\})$ have the desired properties.
\end{proof}

\begin{proof}[Proof of Theorem \ref{iterationtheorem}]
Let $\PP=\standarditeration$ be an iteration of $f$-preserving forcings which preserve old $f$-stationary sets and forces $\SRP$ at successor steps. We may assume inductively that $\PP_\alpha$ preserves $f$ for all $\alpha<\gamma$. The successor step is trivial, so we may restrict to $\gamma\in\Lim$. Note that we may further assume that $(\ddagger)$ holds in $V$, otherwise we could work in $V^{\PP_1}$. Let $p\in \PP$ and let $I\coloneqq I^{\QQ_0}_{p(0)}$. $I$ is a normal uniform $f$-ideal by Proposition \ref{fidealprop}. Now let $\theta$ be sufficiently large and regular, $X\prec H_\theta$ countable with
\begin{enumerate}[label=$(X.\roman*)$]
\item $\PP, p, f\in X$ and
\item $X$ respects $I$.
\end{enumerate}
By Proposition \ref{trueprojectiveprop} and Lemma \ref{fpreservinglemm}, it suffices to find $q\leq p$ that is $(X, \PP, f)$-semigeneric. Note that $X$ is $\fgood$ as $I$ is a $f$-ideal. Let 
$$h\colon \omega\rightarrow \omega\times\omega$$
be a surjection with $i\leq n$ whenever $h(n)=(i, j)$.\\
Let $\delta$ denote $\delta^X$. We will construct a fusion structure 
$$T, \langle \msup{p}{a, n}, \msup{T}{a, n}\mid a\in T_n,\ n<\omega\rangle$$
in $\PP$ as well as names  
$$\left\langle \msup{\dot X}{a, n}, \msup{\dot Z}{a, n} \left(\msup{\dot D_j}{a, n}\right)_{j<\omega}, \msup{\dot I}{a, n}\mid a\in T_n, n<\omega\right\rangle$$
so that for any $n<\omega$ and $a\in T_n$
\begin{enumerate}[label=$(F.\roman*)$]
\item\label{Qitlemmstartcond1}  $T_0=\{\1\}$, $\msup{p}{\1, 0}=p$, $\msup{\dot X}{\1, 0}=\check X$,  $\msup{\dot I}{\1, 0}=\check I$,
\item\label{Qitlemmstartcond2} $\msup{T}{\1, 0}\in X$ is a nested antichain that $p$ is a mixture of with $\msup{T_0}{\1, 0}=\{\1\}$,
\item\label{Qitlemmzcond} $a\forces_{\lh(a)}\msup{\dot Z}{a, n}=\msup{\dot X}{a, n}\cap V$,
\item\label{Qitlemmstartcond3} $\left(\msup{\dot D_j}{a, n}\right)_{j<\omega}$ is forced by $a$ to be an enumeration of all dense subsets of $\pi_{\msup{\dot Z}{a, n}}^{-1}(\check\PP)$ in 
$$M_{\msup{\dot Z}{a, n}}\left[\widecheck{f(\delta)}\right],$$
\item\label{Qitlemmstrongercond} $a\leq \msup{p}{a, n}\res\lh(a)$,
\item\label{Qnotalimcond} $\lh(a)$ is not a limit ordinal,
\item\label{Qitlemmindotxcond} $a\forces_{\lh(a)}\check{p}^{(a, n)}, \check{T}^{(a, n)}, \dot G_{\lh(a)}\in \msup{\dot X}{a, n}$,
\item\label{QdotIcond} $a\forces_{\lh(a)}\msup{\dot I}{a, n}=I^{\Q_{\lh(a)}}_{\msup{\check p}{a, n}(\lh(a))}$ and
\item\label{Qitlemmgoodcond} $a\forces ``\check X\sqsubseteq \msup{\dot X}{a, n}\prec H_{\check\theta}^{V[\dot G_{\lh(a)}]}\text{ is countable and respects }\msup{\dot I}{a, n}"$.
\end{enumerate}
Moreover, for any $b\in\sucm^n_T(a)$
\begin{enumerate}[label=$(F.\roman*)$, resume]
\item\label{Qitlemmoldelemcond} $b\res\lh(a)\forces_{\lh(a)}`` \msup{\check p}{b, n+1}, \msup{\check T}{b, n+1}\in\msup{\dot X}{a, n},\text{ in particular }\lh(\check b), \PP_{\lh(\check a), \lh(\check b)}\in\msup{\dot X}{a, n}$",
\item\label{Qitlemmnextcond} $b\forces_{\lh(b)} \msup{\dot X}{a, n}[\dot G_{\lh(a), \lh(b)}]\sqsubseteq \msup{\dot X}{b, n+1}$ and
\item\label{Qitlemmhitcond} if $h(n)=(i, j)$ and $c=\pred_T^{i}(b)$ then
$$b\forces_{\lh(n)}\msup{\check p}{b, n+1}\in \pi_{\msup{\dot X}{a, n}}[\dot \mu_{c, a}^+(\msup{\dot D_j}{c, i})].$$
\end{enumerate}

Here, $\mu_{c, a}^+$ denotes\footnote{There is some slight abuse of notation here in an effort to improve readability.} 
$$\mu_{\msup{\dot Z}{c, i}, \msup{\dot Z}{a, n}}^+\colon M_{\msup{\dot Z}{c, i}}[\check f(\check\delta)]\rightarrow M_{\msup{\dot Z}{a, n}}[\check f(\check \delta)].$$
We define all objects by induction on $n<\omega$. 
$$T_0=\{\1\}, \msup{p}{\1, 0}, \msup{T}{\1, 0}, \msup{\dot X}{\1, 0},\msup{\dot Z}{\1, 0} \left(\msup{\dot D_j}{\1, 0}\right)_{j<\omega}, \msup{\dot I}{\1, 0}$$
are given by \ref{Qitlemmstartcond1}-\ref{Qitlemmstartcond3} and \ref{QdotIcond}. Suppose we have already defined
$$T_n, \left\langle \msup{p}{a, n}, \msup{T}{a, n}, \msup{\dot X}{a, n},\msup{\dot Z}{a, n}, \left(\msup{\dot D_j}{a, n}\right)_{j<\omega}\mid a\in T_n\right\rangle$$
and we will further construct 
$$T_{n+1}, \left\langle \msup{p}{b, n+1}, \msup{T}{b, n+1}, \msup{\dot X}{b, n+1},\msup{\dot Z}{b, n+1}, \left(\msup{\dot D_j}{b, n+1}\right)_{j<\omega}\mid b\in T_{n+1}\right\rangle.$$
Fix $a\in T_n$. Let $E$ be the set of all $b$ so that

\begin{enumerate}[label=$(E.\roman*)$]
\item $b\in \PP_{\lh(b)}$ and $\lh(b)<\gamma$,
\item $\lh(a)\leq\lh(b)$ and $b\res\lh(a)\leq a$,
\end{enumerate}
and there are a nested antichain $S$ in $\PP$, $s\in\PP$ and names $\dot X$, $\dot I$ with
\begin{enumerate}[label=$(E.\roman*)$, resume]
\item $S\hooks\msup{T}{a, n}$,
\item $s\leq\msup{p}{a, n}$ is a mixture of $S$,
\item\label{QitlemmEgencond} if $h(n)=(i, j)$ and $c=\pred_T^i(a)$ then
$$b\forces_{\lh(b)} \check s\in \pi_{\msup{\dot Z}{a, n}}[\dot\mu^+_{c, a}(\msup{\dot D_j}{c, i})],$$
\item $\lh(b)$ is not a limit ordinal,
\item $b\res\lh(a)\forces_{\lh(a)}\check s, \check S \in\dot X$,
\item $b\forces_{\lh(b)} \check s\res\lh(b)\in\dot G_{\lh(b)}$,
\item\label{QitlemmEsqsubcond} $b\forces_{\lh(b)}\msup{\dot X}{a, n}\sqsubseteq\msup{\dot X}{a, n}[\dot G_{\lh(a), \lh(b)}]\sqsubseteq \dot X\prec H_{\check\theta}^{V[\dot G_{\lh(\check b)}]}$,
\item $b\forces_{\lh(b)}``\dot X\text{ is countable and respects }\dot I"$,
\item $b\forces_{\lh(b)}\dot I=I^{\Q_{\lh(b)}}_{\check s(\lh(b))}$ and
\item  if $S_0=\{c_0\}$ then $\lh(b)=\lh(c_0)$ and $b\leq c_0$. 
\end{enumerate}

\begin{claim}
$E\res\lh(a)\coloneqq \{b\res\lh(a)\mid b\in E\}$ is dense in $\PP_{\lh(a)}$.
\end{claim}

\begin{proof}
Let $a'\leq a$ and let $G$ be $\PP_{\lh(a)}$-generic with $a'\in G$. 
By \ref{Qitlemmstrongercond}, $\msup{p}{a, n}\res\lh(a)\in G$. Work in $V[G]$. Let $h(n)=(i, j)$ and $c=\pred_T^i(a)$. Let
$$\msup{X}{c, i}=\left(\msup{\dot X}{c, i}\right)^{G_{\lh(c)}}\text{ and }\msup{X}{a, n}=\left(\msup{\dot X}{a, n}\right)^G$$
as well as $\msup{Z}{c, i}=\msup{X}{c, i}\cap V$, $\msup{Z}{a, n}=\msup{X}{a, n}\cap V$. Find $r\in \msup{T_1}{a, n}$ with $r\res\lh(a)\in G$. As $\msup{p}{a, n}$ is a mixture of $\msup{T}{a, n}$, we have 
$$r\leq \msup{p}{a, n}\res\lh(r).$$ 
Let $\hat r=r^\frown\msup{p}{a, n}\res[\lh(r), \gamma)$. Note that $\hat r\in \msup{X}{a, n}$, as 
$$\msup{p}{a, n}, \msup{T}{a, n}, G\in\msup{X}{a, n}$$
by \ref{Qitlemmindotxcond}. Moreover, $\hat r\res\lh(a)\in G$. Let $\QQ\coloneqq\Q_{\lh(a)}^G$ and 
$$D\coloneqq\mu_{c, a}^+((\dot D^i_j)^{G_{\lh(c)}})\in M_{\msup{Z}{a, n}}[f(\delta)]\subseteq M_{\msup{X}{a, n}}[f(\delta)].$$

\begin{sub}
There are $s$, $Y$ with
\begin{enumerate}[label=$(\roman*)$]
\item $\msup{X}{a, n}\sqsubseteq Y\prec H_\theta^{V[G]}$,
\item $s\leq \msup{p}{a, n}$,
\item $s\res\lh(a)\in G$,
\item $s\in \pi_Y[\mu_{\msup{X}{a, n}, Y}^+(D)]$ and
\item\label{Qsubtruecond} $Y$ respects $I^{\QQ}_{s(\lh(a))}$.
\end{enumerate}
\end{sub}

\begin{proof}
Let
$$D_0\coloneqq\{t\in D\mid \pi_{\msup{X}{a, n}}(t)\leq\msup{p}{a, n}\wedge \pi_{\msup{X}{a, n}}(t)\res\lh(a)\in G\}$$
and $D_1$ be the projection of $D_0$ onto $\pi_{\msup{X}{a, n}}^{-1}(\QQ)$. Observe that 
$$M_{\msup{X}{a, n}}[f(\delta)]\models``D_1\text{ is dense below } \pi_{\msup{X}{a, n}}^{-1}(\msup{p}{a, n}(\lh(a))\text{ in }\pi_{\msup{X}{a, n}}^{-1}(\QQ)".$$
Applying Lemma \ref{technicalQlemm} with (making use of the notation there) 
\begin{itemize}
\item $\PP=\Q$,
\item $p=\msup{p}{a, n}(\lh(a))$,
\item $X=\msup{X}{a, n}$ and
\item $D=D_0$,
\end{itemize} 
we find some countable $Y$ and some $s_0$ with 
\begin{enumerate}[label=\rn*]
\item $\msup{X}{a, n}\sqsubseteq Y\prec H_\theta^{V[G]}$,
\item $s_0\leq \msup{p}{a, n}(\lh(a))$,
\item $s_0\in \pi_{Y}[\mu_{\msup{X}{a, n}, Y}^+(D_1)]$ and
\item $Y$ respects $I^\QQ_{s_0}$.
\end{enumerate}

By definition of $D_1$, there is $s\leq\msup{p}{a, n}$ with 
\begin{enumerate}[label=$(s.\roman*)$]
\item $s\res\lh(a)\in G$,
\item $s\in\pi_Y[\mu_{\msup{X}{a, n}, Y}^+(D)]$ and
\item $s(\lh(a))=s_0$.
\end{enumerate}
$Y, s$ have the desired properties. 

\end{proof}
We can now apply Fact \ref{m2.11fact} in $Y$ and get a nested antichain $S\in\msup{X}{a, n}$ with 
\begin{enumerate}[label=$(S.\roman*)$]
\item $s$ is a mixture of $S$,
\item if $S_0=\{d\}$ then $\lh(r)\leq\lh(d)$, $d\res\lh(r)\leq r$ and $\lh(d)$ is not a limit ordinal and
\item $S\hooks \msup{T}{a, n}$.
\end{enumerate}
Let $\dot X$ be a name for $Y[\dot G_{\lh(a),\lh(d)}]$ and $\dot I$ a name for $I^{\Q_{\lh(d)}}_{s(\lh(d))}$.

\begin{sub}
In $V[G]$, we have
$$\dot I^{s\res\lh(d)}=I^{\PP_{\lh(a),\lh(d)+1}}_{s\res\lh(d)+1}=I^\QQ_{s(\lh(a))}.$$
\end{sub}
\begin{proof}
The first equality is simply by definition of $\dot I$. The second equality follows as we preserve old $f$-stationary sets along the iteration and since $\PP_{\lh(a),\lh(d)+1}$ preserves $f$ by our inductive hypothesis.
\end{proof}

It follows that
$$Y\text{ respects }\dot I^{s\res\lh(d)}.$$
As $\lh(a)$ is not a limit ordinal, $(\ddagger)$ holds in $V[G]$, so that $\PP_{\lh(a),\lh(d)}$ is respectful by Lemma \ref{sledgehammerlemm}. Thus there is $b\in \PP_{\lh(a), \lh(d)}$, $b\leq s\res\lh(d)$ so that 
$$b\forces_{\lh(b)}``\check Y\sqsubseteq \check Y[\dot G]\text{ respects }\dot I".$$
Since $b\res\lh(a)\in G$, we may assume further that $b\res\lh(a)\leq a'$. $s, S, \dot X, \dot I$ witness $b\in E$.
\end{proof}

To define $T_{n+1}$, fix a maximal antichain $A\subseteq E\res\lh(a)$, and for any $e\in A$ choose $b_e\in E$ with $b_e\res\lh(a)=e$.  We set $\sucm_T^n(a)=\{b_e\mid e\in A\}$. For any $b\in \sucm_T^n(a)$, let $S, s, \dot X, \dot I$ witness $b\in E$. We then let 
\begin{itemize}
\item$\msup{p}{b, n+1}=s,\ \msup{T}{b, n+1}=S,\ \msup{\dot X}{b, n+1}=\dot X,\ \msup{\dot I}{b, n+1}=\dot I$,
\item $\msup{\dot Z}{b, n+1}$ be a name for $\dot X\cap V$ and
\item $\left(\msup{\dot D_j}{b, n+1}\right)_{j<\omega}$ be a sequence of names that are forced by $b$ to enumerate all dense subsets of $\pi_{\msup{\dot Z}{b, n+1}}^{-1}(\PP)$ in $M_{\msup{\dot Z}{b, n+1}}\left[\widecheck{f(\delta)}\right]$.
\end{itemize}
This finishes the construction.\medskip

By Fact \ref{m2.11fact}, there is a mixture $q$ of $T$. Let $G$ be $\PP$-generic with $q\in T$. By Fact \ref{m3.5fact}, in $V[G]$ there is a sequence $\langle a_n\mid n<\omega\rangle$ so that for all $n<\omega$
\begin{enumerate}[label=$(\vec a.\roman*)$]
\item $a_0=q_0$,
\item $a_{n+1}\in\sucm_T^{n}(a_n)$ and
\item $\msup{p}{a_n, n}\in G$.
\end{enumerate}  

For $n<\omega$, let $\alpha_n=\lh(a_n)<\gamma$. For $n<\omega$ we let 
$$X_n\coloneqq \left(\msup{\dot X}{a_n, n}\right)^{G_{\alpha_n}}$$
and also 
$$X_\omega=\bigcup_{n<\omega} X_n[G_{\alpha_n, \gamma}].$$
Further, for $n\leq \omega$ let 
$$Z_n\coloneqq X_n\cap V\text{ and }\pi_n\coloneqq \pi_{Z_n}.$$
We remark that 
$$X_n[G_{\alpha_n, \gamma}]\sqsubseteq X_m[G_{\alpha_m,\gamma}]\prec H_\theta^{V[G]}$$
follows inductively from \ref{Qitlemmindotxcond} and \ref{Qitlemmgoodcond} for $n\leq m<\omega$ so that $X_\omega\prec H_\theta^{V[G]}$. We aim to prove that 
$$X\sqsubseteq X_\omega\text{ is }\fgood.$$
In fact, we will show
\begin{enumerate}[label=$(Z_\omega.\roman*)$]
\item\label{QZcond1} $X\sqsubseteq Z_\omega$,
\item\label{QZcond2} $Z_\omega\text{ is }\fgood$ and
\item\label{QZcond3} $\pi_{\omega}^{-1}[G]\text{ is generic over }M_{\omega}[f(\delta)]$,
\end{enumerate}
which implies the above.

\begin{claim}\label{Qeqaulityclaim}
$Z_\omega=\bigcup_{n<\omega} Z_n$.
\end{claim}

\begin{proof}
$``\supseteq"$ is trivial, so we show $``\subseteq"$. Let $x\in Z_\omega$ and find $i<\omega$ with $x\in X_i[G_{\alpha_i, \gamma}]$. Note that there is $\dot x\in Z_i$ a $\PP$-name for a set in $V$ with $x=\dot x^G$. Let $D\in M_i$ be the dense set of conditions in $\pi_n^{-1}(\PP)$ deciding $\pi_i^{-1}(\dot x)$. There must be some $j<\omega$ so that 
$$\left(\msup{\dot D_j}{a_i, i}\right)^G =D.$$
Now find $n$ with $h(n)=(i, j)$. We then have
$$\msup{p}{a_{n+1}, n+1}\in \pi_n[\mu_{a_i, a_{n+1}}^+(D)]$$
by \ref{Qitlemmhitcond}. We have that $\msup{p}{a_{n+1}, n+1}$ decides $\dot x$ to be some $z\in X_n$, and as $\msup{p}{a_{n+1}, n+1}\in G$,
$$x=\dot x^G=z\in X_n\cap V=Z_n.$$
\end{proof}

As $X\sqsubseteq X_n$ is $\fgood$ by \ref{Qitlemmgoodcond} for $n<\omega$, \ref{QZcond1} and \ref{QZcond2} follow at once. It remains to show \ref{QZcond3}.\\
As $Z_\omega$ is $\fgood$ and by Claim \ref{Qeqaulityclaim}, we have that
$$\langle M_\omega[f(\delta)], \mu_{n,\omega}^+\mid n<\omega\rangle=\varinjlim\langle M_n[f(\delta)], \mu_{n, m}^+\mid n\leq m<\omega\rangle$$
for some $(\mu_{n,\omega}^+)_{n<\omega}$. Let $E\in M_\omega[f(\delta)]$ be dense in $\pi_{\omega}^{-1}(\PP)$. Then for some $i,j<\omega$, $E=\mu_{i,\omega}^+(D)$ for 
$$D\coloneqq \left(\msup{\dot D_j}{a_i, i}\right)^G.$$
Find $n$ with $h(n)=(i, j)$. By \ref{Qitlemmhitcond},
$$\msup{p}{a_{n+1}, n+1}\in \pi_{n}[\mu_{i, n}^+(D)]\subseteq\pi_{\omega}[\mu_{i,\omega}^+(D)]=\pi_\omega[E].$$
As $\msup{p}{a_{n+1}, n+1}\in G$, we have $E\cap \pi_\omega^{-1}[G]\neq\emptyset$, which is what we had to show.
\end{proof}

\section{$f$-Proper and $f$-Semiproper Forcings}

Suppose $f$ witnesses $\diab$. We already used the term $(X, \PP, f)$-semigeneric which suggests there should be a notion of $f$-semiproperness. Indeed there is and it behaves roughly like semiproperness. In fact, there are several other classes associated to $f$ which mirror well-known forcing classes.

\begin{defn}
    A forcing $\PP$ is $f$-complete if for any sufficiently large regular $\theta$, for any $\fgood$ $X\prec H_\theta$ with $\PP\in X$ and any $g\subseteq \bar\PP$ generic over $M_X[f(\delta^X)]$, there is a some $p\in \PP$ with 
    $$p\forces\dot G\cap \check X=\pi_X[\check g].$$
\end{defn}
\begin{defn}
        A forcing $\PP$ is $f$-proper if for any sufficiently large regular $\theta$, any $\fgood$ $X\prec H_\theta$ with $\PP\in X$ and any $p\in X\cap \PP$, there is a $(X, \PP, f)$-generic condition $q\leq p$, that is a condition $q$ with 
        $$q\forces\dot G\cap\check X\text{ is generic over }\check X\wedge \check X[\dot G]\text{ is }\check f\text{-slim}".$$
\end{defn}
\begin{defn}
    A forcing $\PP$ is $f$-semiproper if for any sufficiently large regular $\theta$, any $\fgood$ $X\prec H_\theta$ with $\PP\in X$ and any $p\in X\cap \PP$, there is a $(X, \PP, f)$-semigeneric condition $q\leq p$.
\end{defn}

The following graphic collects all provable relations between the relevant forcing classes.

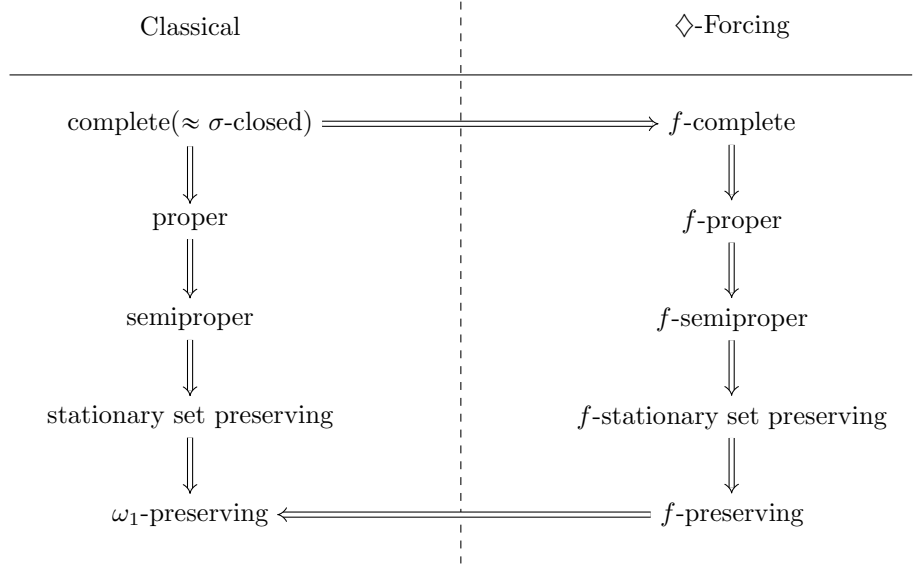
\begin{center}
\begin{tikzpicture}
\def\y{0.65};
\def\x{1.2};

\node (classic) at (0, 2*\y) {Classical};
\node (diamond) at (6*\x, 2*\y) {$\diamondsuit$-Forcing};

\node (sigma) at (0, 0){complete($\approx\sigma$-closed)};
\node (proper) at (0, -2*\y) {proper};
\node (semiproper) at (0, -4*\y) {semiproper};
\node (stationary) at (0, -6*\y) {stationary set preserving};
\node (omega1pres) at (0, -8*\y) {$\omega_1$-preserving};
\node (fcomp) at (6*\x, 0) {$f$-complete};
\node (fproper) at (6*\x, -2*\y) {$f$-proper};
\node (fsemiproper) at (6*\x, -4*\y) {$f$-semiproper};
\node (fstationary) at (6*\x, -6*\y) {$f$-stationary set preserving};
\node (fpres) at (6*\x, -8*\y) {$f$-preserving};

\draw[-] (-2*\x, 1*\y) -- (8*\x, 1*\y);
\draw[-, dashed] (3*\x, 2.5*\y) -- (3*\x, -9*\y);
\draw[-implies, double equal sign distance] (sigma)--(proper);
\draw[-implies, double equal sign distance] (proper)--(semiproper);
\draw[-implies, double equal sign distance] (semiproper)--(stationary);
\draw[-implies, double equal sign distance] (stationary)--(omega1pres);
\draw[-implies, double equal sign distance] (fcomp)--(fproper);
\draw[-implies, double equal sign distance] (fproper)--(fsemiproper);
\draw[-implies, double equal sign distance] (fsemiproper)--(fstationary);
\draw[-implies, double equal sign distance] (fstationary)--(fpres);

\draw[-implies, double equal sign distance] (sigma)--(fcomp);
\draw[-implies, double equal sign distance] (fpres)--(omega1pres);

\end{tikzpicture}
\end{center}

We also get the expected iteration theorems.

\begin{thm}
    Any countable support iteration of $f$-complete (resp. $f$-proper) forcings is $f$-complete (resp. $f$-proper).
\end{thm}

\begin{thm}
    Any nice iteration of $f$-semiproper forcings is $f$-semiproper.
\end{thm}

The proof is much easier than that of Theorem \ref{iterationtheorem}, so we omit it.

\bibliographystyle{alpha}
\bibliography{bib}
\end{document}